\documentclass[12pt,a4paper]{article}%
\usepackage[active]{srcltx}
\usepackage[utf8]{inputenc}
\usepackage{amsmath}
\usepackage{amsfonts}
\usepackage{amssymb}
\usepackage{graphicx}%
\providecommand{\U}[1]{\protect\rule{.1in}{.1in}}
\newtheorem{theorem}{Theorem}

\newtheorem{corollary}[theorem]{Corollary}

\newtheorem{definition}[theorem]{Definition}
\newtheorem{example}[theorem]{Example}

\newtheorem{lemma}[theorem]{Lemma}

\newtheorem{proposition}[theorem]{Proposition}
\newtheorem{remark}[theorem]{Remark}

\newenvironment{proof}[1][Proof]{\noindent\textbf{#1.} }{$\hfill\Box$}
\begin{document}

\title{A quasilinear problem in two parameters depending on the gradient
{}}
\author{\textbf{{\large H. Bueno and G. Ercole}}\thanks{The authors were supported in
part by FAPEMIG and CNPq, Brazil.}\\\textit{{\small Departamento de Matem\'{a}tica}}\\\textit{{\small Universidade Federal de Minas Gerais}}\\\textit{{\small Belo Horizonte, Minas Gerais, 30.123.970, Brazil}}\\\textit{{\small e-mail: hamilton@mat.ufmg.br}}, \thinspace
\textit{{\small grey@mat.ufmg.br}}}
\maketitle

\begin{abstract}
The existence of positive solutions is considered for the Dirichlet problem
\[
\left\{
\begin{array}
[c]{rcll}%
-\Delta_{p}u & = & \lambda\omega_{1}(x)\left\vert u\right\vert ^{q-2}%
u+\beta\omega_{2}(x)\left\vert u\right\vert ^{a-1}u|\nabla u|^{b} & \text{in
}\Omega\\
u & = & 0 & \text{on }\partial\Omega,
\end{array}
\right.
\]
where $\lambda$ and $\beta$ are positive parameters, $a$ and $b$ are positive
constants satisfying $a+b\leq p-1$, $\omega_{1}(x)$ and $\omega_{2}(x)$ are
nonnegative weights and $1<q\leq p$. The homogeneous case $q=p$ is handled by
making $q\rightarrow p^{-}$ in the sublinear case $1<q<p,$ which is based on
the sub- and super-solution method. The core of the proof of this problem is
then generalized to the Dirichlet problem $-\Delta_{p}u=f(x,u,\nabla u)$ in
$\Omega$, where $f$ is a nonnegative, continuous function satisfying simple,
geometrical hypotheses. This approach might be considered as a unification of
arguments dispersed in various papers, with the advantage of handling also
nonlinearities that depend on the gradient, even in the $p$-growth case. It is
then applied to the problem $-\Delta_{p}u=\lambda\omega(x)u^{q-1}\left(
1+|\nabla u|^{p}\right)  $ with Dirichlet boundary conditions in the domain
$\Omega$.

\end{abstract}


\textit{keywords:} $p$-Laplacian, positive solution, nonlinearity depending on
the gradient, $p$-growth, sub- and super-solution method.

\section{Introduction}

Existence of positive solutions for $p$-Laplacian problems depending on the
gradient has been attracting considerable interest among researchers of
elliptic PDE's, but no general method to deal with this kind of problem has
been established. The dependence on the gradient requests a priori bounds on
the solutions and in their derivatives, what brings additional difficulties.
In general, this problem is not suitable for variational techniques and thus
topological methods (as fixed-point or degree results) and/or blow-up
arguments are normally employed to solve it (\cite{figubilla,iturriaga,ruiz}).

In the case of the Laplacian (i.e., $p=2$) an interesting combination of
variational and topological techniques (precisely, a combination of the
mountain pass geometry with the contraction lemma) was first used in
\cite{djairo} and has motivated some works (e.g., \cite{ubilla}). Basically,
an iteration process is constructed by freezing the gradient in each iteration
and (variationally) solving the resulting problem. Then, Lipschitz hypotheses
in the variables $u$ and $v$ are made on $f(x,u,v)$ in order to guarantee the
convergence in $W_{0}^{1,2}(\Omega)$ of the obtained sequence of solutions.
The same approach for the $p$-Laplacian with $p>2$ seems to be not directly
adaptable, since the natural extension of the Lipschitz conditions to obtain
the convergence of the iterated solutions leads $f$ to be a H\"{o}lder
function with exponent greater than $1$ in both variables $u$ and $v$.

In \cite{figubilla}, the authors discuss the existence of positive solutions
for quasilinear elliptic equations in annular domains in $\mathbb{R}^{N}$ and,
in particular, the radial Dirichlet problem in annulus. (Therefore, the
problem is transformed into an ordinary differential equation.) In that paper,
$f$ satisfies a superlinear condition at $0$ and a local superlinear condition
at $+\infty$. The growth of the nonlinearity $f$ in relation to the gradient
is controlled by a Bernstein-Nagumo condition and a local homogeneity type
condition in the second variable. The existence of solutions is guaranteed by
applying the Krasnosel'skii Fixed Point Theorem for mappings defined in cones.

In this paper we consider the existence of positive solutions for the
Dirichlet problem in two parameters in a smooth, bounded domain $\Omega
\subset\mathbb{R}^{N}$ ($N>1$):
\begin{equation}
\left\{
\begin{array}
[c]{rcll}%
-\Delta_{p}u & = & \lambda\omega_{1}(x)\left\vert u\right\vert ^{q-2}%
u+\beta\omega_{2}(x)\left\vert u\right\vert ^{a-1}u|\nabla u|^{b} & \text{in
}\Omega\\
u & = & 0 & \text{on }\partial\Omega,
\end{array}
\right.  \label{two}%
\end{equation}
where $\Delta_{p}u:=\operatorname{div}\left(  |\nabla u|^{p-2}\nabla u\right)
$ is the $p$-Laplacian operator for $p>1$, $\lambda$ and $\beta$ are positive
parameters, $a$ and $b$ are positive constants satisfying $a+b\leq p-1$,
$\omega_{1}(x)$ and $\omega_{2}(x)$ are nonnegative weights and $1<q\leq p$.

By applying the sub- and super-solution method for problems involving the
gradient (\cite{Boccardo,mabel}), we treat first the sublinear case $1<q<p$
and, in Theorem \ref{taux}, we prove the existence of at least one positive
solution $u\in C^{1,\tau}(\overline{\Omega})$.

The case $q=p$ is more demanding and our approach makes $q\rightarrow p^{-}$
in the solution obtained in Theorem \ref{taux}. Our result is stated in
Theorem \ref{mauxt}.

As a consequence of the study of problem (\ref{two}), we realized that the
core of the proof of Theorem \ref{taux} could be generalized to handle the
Dirichlet problem
\begin{equation}
\left\{
\begin{array}
[c]{rcll}%
-\Delta_{p}u & = & f(x,u,\nabla u) & \text{in }\Omega\\
u & = & 0 & \text{on }\partial\Omega,
\end{array}
\right.  \label{prob}%
\end{equation}
$f$ is a nonnegative, continuous function satisfying simple hypotheses and
$\Omega\subset\mathbb{R}^{N}$ is a bounded, smooth domain, $N>1$.

Our proof of existence of a positive solution for (\ref{prob}) might be
considered as a unification of arguments dispersed in various papers, but it
also handles nonlinearities that depend on the gradient in the $p$-growth
case, which is remarkable. As in our treatment of problem (\ref{two}), it is a
consequence of the sub- and super-solution method for quasilinear equations
involving dependence on the gradient. To apply the method, a condition of the
Bernstein-Nagumo type is always assumed; in \cite{Boccardo} the nonlinearity
$f$ is a Carath\'{e}odory function (i.e., measurable in the $x$-variable and
continuous in the $(u,v)$-variable) such that

\begin{enumerate}
\item[(H1)] $f(x,u,v)\leq C(|u|)(1+|v|^{p})\quad(u,v)\in\mathbb{R}%
\times\mathbb{R}^{N},\ { a.e. }\ x\in\Omega$ for some increasing function
$C\colon[0,\infty]\to[0,\infty]$.\label{Boccardo}
\end{enumerate}

This assumption is merely technical and can be chosen as any hypothesis that
guarantees the existence of a solution of (\ref{prob}) from an ordered sub-
and super-solution pair. We have taken for granted the growth condition (H1),
following \cite{Boccardo}. Since this condition is also related to the
regularity of a weak solution, it is by no means surprisingly that assumptions
of the same type are also found in papers that do not apply the sub- and
super-solution method.

Most of the articles dealing with sub- and super-solution method for problems
with the $p$-Laplacian and nonlinearity depending on the gradient aim to
improve the method itself and applications of the method are rare. One of the
exceptions is the work of Grenon \cite{Grenon}, where problem (\ref{prob}) --
with different hypothesis -- is solved by analyzing two symmetrized problems.
From the existence of two nontrivial super-solutions $V_{1}$ and $V_{2}$ for
those problems follows the existence a super-solution $U_{1}$ and a
sub-solution $U_{2}$ for $(\ref{prob})$, with $U_{2} \leq U_{1}$. (The article
\cite{BEFZ} also applies the sub- and super-solution method for a nonlinearity
$f$ depending on the gradient. However, the obtention of a sub-solution
follows a quite different method.)

The work \cite{ILS} also applies the sub- and super-solution method and deals
with a problem that depends on the gradient of the solution in a special form.
However, by applying a change of variable, the problem is transformed into one
that does not depends on the gradient and the usual method of sub- and
super-solution is then applied. (See also Example \ref{ex2}.)

Contrasting with the majority of papers on the subject -- in which many
hypotheses are normally assumed on the nonlinearity -- our assumptions on $f$
are very simple. Besides (H1), they consist of hypotheses (H2) and (H3), which
we now describe.

Let $\omega\neq0$ be a continuous, nonnegative function and $\lambda_{1}$ be
the first eigenvalue of the Dirichlet problem for $-\Delta_{p}$ with weight
$\omega$ in the domain $\Omega$, that is, $\lambda_{1}$ is the least positive
number such that
\begin{equation}
\label{1eigenv}\left\{
\begin{array}
[c]{rcll}%
-\Delta_{p}u & = & \lambda_{1}\omega u^{p-1} & \text{in }\Omega\\
u & = & 0 & \text{on }\partial\Omega,
\end{array}
\right.
\end{equation}
for some $u\in W^{1,p}_{0}(\Omega)$, $u>0$ in $\Omega$. Our assumptions on the
nonlinearity $f$ depend on the chosen weight function $\omega$.

Our assumption (H2) is a standard sublinear condition: near $u=0^{+}$ the
values of the nonlinearity $f(x,u,v)$ must be greater than $\lambda_{1}%
u^{p-1}\omega(x)$. We show that this assumption produces, for each
$\epsilon>0$, a positive sub-solution $u_{\epsilon}$, whose sup norm becomes
small when $\epsilon$ decreases. If the nonlinearity $f$ depends only on
$(x,u)$, this is a well known fact; to our knowledge, if the nonlinearity $f$
depends on $(x,u,\nabla u)$, this fact was overlooked in previous papers.

The last assumption, (H3), is that $f$, restricted to a suitable compact set,
is bounded from above by a special multiple of the weight $\omega$. This
approach follows \cite{bez}, where (\ref{prob}) was also independent of the
gradient. We show that this hypothesis produces a super-solution $U$ for
(\ref{prob}) with $u_{\epsilon}<U$, for $\epsilon$ small enough.

The paper is organized as follows. In Section \ref{Basic}, we summarize some
results about the $p$-Laplacian with Dirichlet boundary conditions.

In Section \ref{tp} we prove our results about problem (\ref{two}), that is,
Theorems \ref{taux} and \ref{probp}.

In Section \ref{main} we state and prove our main result about the abstract
problem (\ref{prob}), that is, Theorem \ref{maint}.

In the special case $f(x,u,v)=\omega(x)g(u,|v|)$, with $\omega>0$ continuous,
hypotheses (H2)-(H3) are interpreted in Section \ref{App}. There, we also
consider problem (\ref{prob}) for the parameter dependent nonlinearity
\[
f(x,u,v)=\lambda\omega(x)u^{q-1}\left(  1+|\nabla u|^{p}\right)  .
\]
We prove the existence of $\lambda^{*}>0$ such that, in this case,
(\ref{prob}) has a solution for all $\lambda\in[0,\lambda^{*}]$. Observe that
no restriction on the value of $p>1$ or on the dimension $N$ is assumed.

In this example, our result follows directly from Theorem \ref{maint}, but the
arguments are standard and can be found, for nonlinearities of the type
\[
\lambda u^{q-1}+g(u)
\]
and $1<q<p$, \textit{e. g.} in the articles \cite{GuoZhang,GuoWebb1}. (With a
different method, the same problem is considered in \cite{GMA}.)

\section{Preliminaries}

\label{Basic}

In this section we recall some basic results in the theory of the
$p$-Laplacian equation with Dirichlet boundary condition and present technical
results that will be used in the rest of the paper. Let $\Omega$ be a bounded,
smooth domain in $\mathbb{R}^{N}$, $N>1$.

\begin{definition}
\label{defsubsupsol}Let $f\colon\Omega\times\mathbb{R}\times\mathbb{R}^{N}$ be
a Carath\'{e}odory function. A function $u\in W^{1,p}(\Omega)\cap\,L^{\infty
}(\Omega)$ is called a solution $($sub-solution, super-solution$)$ of
\begin{equation}
{\displaystyle\left\{
\begin{array}
[c]{rcll}%
-\Delta_{p}u & = & f(x,u,\nabla u) & \mbox{in}\ \Omega,\\
u & = & 0 & \mbox{on}\ \partial\Omega,
\end{array}
\right.  } \nonumber\label{p1}%
\end{equation}
if
\[
\displaystyle{\int_{\Omega}|\nabla u|^{p-2}\nabla u\cdot\nabla\phi
\,dx=\int_{\Omega}f(x,u,\nabla u)\phi\,dx\ \ (\leq0,\ \geq0),}%
\]
for all $\phi\in C_{0}^{\infty}(\Omega)$, $(\phi\geq0$ in $\Omega$ in the case
of a sub- or super-solution$)$ and
\[
u=0\ (\leq0,\ \geq0)\ \ \text{on}\ \ \partial\Omega.
\]

A pair $(\underline{u},\overline{u})$ of sub- and super-solution is
\textrm{ordered} if $\underline{u}\leq\overline{u}$ a.e.\goodbreak

\end{definition}

We remark that, if the nonlinearity $f$ satisfies (H1), then
\begin{align*}
\int_{\Omega}\big|f(x,u,\nabla u)\phi\big|dx  &  \leq C(\Vert u\Vert_{\infty
})\int_{\Omega}(1+|\nabla u|^{p})|\phi|\,dx\\
&  =C(\Vert u\Vert_{\infty})\left(  \int_{\Omega}|\phi|\,dx+\int_{\Omega
}|\nabla u|^{p}|\phi|\,dx\right)  <\infty,
\end{align*}
since $\phi\in C_{0}^{\infty}$ and $u\in W^{1,p}(\Omega)\cap\,L^{\infty
}(\Omega)$.\vspace{0.3cm}

The following Theorem is a simpler version of a result of Lieberman proved in
\cite[Theorem 1]{LIEBERMAN} by using techniques developed by DiBenedetto
\cite{DiBenedetto} and Tolksdorf \cite{Tolks}.

\begin{theorem}
$\cite[Thm1]{LIEBERMAN}$ Suppose that $u\in W^{1,p}\left(  \Omega\right)  $ is
a weak solution of the Dirichlet problem \label{Liebb}%
\[
\left\{
\begin{array}
[c]{rcll}%
-\Delta_{p}u & = & f(x,u,\nabla u) & \text{ \ \ in }\Omega,\\
u & = & 0 & \text{\ \ \ on }\partial\Omega
\end{array}
\right.
\]
where $f$ is a Carath\'{e}odory function such that%
\[
\left\vert f(x,\xi,\eta)\right\vert \leq\Lambda\left(  1+\left\vert
\eta\right\vert ^{p}\right)  \text{ \ for all }(x,\xi,\eta)\in\Omega
\times\lbrack-M,M]\times\mathbb{R}^{N}%
\]
for positive constants $\Lambda$ and $M.$

If $\left\Vert u\right\Vert _{\infty}\leq M,$ then there exists $0<\alpha<1,$
depending only on $\Lambda,$ $p$ and $N$ such that $u\in C^{1,\alpha}\left(
\overline{\Omega}\right)  ;$ moreover,
\[
\left\Vert u\right\Vert _{1,\alpha}\leq C
\]
where $C$ is a positive constant that depends only on $\Lambda,$ $p,$ $N$ and
$M.$
\end{theorem}

We now state, in a version adapted to our paper, the result that give basis to
the method of sub- and super-solution for equations like (\ref{p1}). The
existence part is a consequence of Thm. 2.1 of Boccardo, Murat and Puel
\cite{Boccardo}. The regularity part follows from Theorem \ref{Liebb}, while
the minimal and maximal solutions are consequence of Zorn's Lemma, as proved
in Cuesta Leon \cite{mabel}:

\begin{theorem}
\label{TBoc} Let $f\colon\Omega\times\mathbb{R}\times\mathbb{R}^{N}%
\to\mathbb{R}$ be a Carath\'{e}odory function satisfying $(\mathrm{H}1)$.
Suppose that $(\underline{u},\overline{u})$ is an ordered pair of sub- and
super-solution for the problem $(\ref{p1})$.

Then, there exists a minimal solution $u$ and a maximal solution $v$ of
$(\ref{p1})$, both in $C^{1,\tau}\left(  \overline{\Omega}\right)  $
$(0<\tau<1)$, such that $\underline{u}\leq u\leq v\leq\overline{u}$.
\end{theorem}

(By \emph{minimal} and \emph{maximal} solution of $(\ref{p1})$ we mean that,
if $w$ is a solution of this problem and $\underline{u}\leq w\leq\overline{u}%
$, then $u\leq w\leq v$.)

\section{A problem involving two parameters}

\label{tp} In this section we consider the existence of positive solutions for
the following problem in two parameters in the smooth, bounded domain
$\Omega\subset\mathbb{R}^{N}$ ($N>1$):%

\begin{equation}
\left\{
\begin{array}
[c]{rcll}%
-\Delta_{p}u & = & \lambda\omega_{1}(x)\left\vert u\right\vert ^{q-2}%
u+\beta\omega_{2}(x)\left\vert u\right\vert ^{a-1}u|\nabla u|^{b} & \text{in
}\Omega\\
u & = & 0 & \text{on }\partial\Omega,
\end{array}
\right.  \label{twotwo}%
\end{equation}
where $\lambda$ and $\beta$ are positive parameters, $a$ and $b$ are positive
constants satisfying $a+b\leq p-1$, $\omega_{1}(x)$ and $\omega_{2}(x)$ are
nonnegative weights and $1<q\leq p$.

In the case $q=p$, we remark that the problem (\ref{twotwo}) is
\emph{homogeneous}, in the sense that if $u$ solves it for fixed parameters
$\lambda$ and $\beta$, then $ku$ is also a solution, for any constant $k$
(note that we are assuming $a+b=p-1$).

In our approach to problem (\ref{twotwo}) the solution $\phi$ of the torsional
creep problem
\begin{equation}
\left\{
\begin{array}
[c]{rcll}%
-\Delta_{p}\phi & = & \omega & \text{in }\Omega\\
\phi & = & 0 & \text{on }\partial\Omega.
\end{array}
\right.  \label{pweight}%
\end{equation}
plays an important role. It is well-known that $\phi\in C^{1,\tau}\left(
\overline{\Omega}\right)  $ and $\phi>0$ in $\Omega$.

We define the positive constants $\alpha=\Vert\phi\Vert_{\infty}^{1-p}$ and
$\mu=\Vert\nabla\phi\Vert_{\infty}/\Vert\phi\Vert_{\infty}$, where $\Vert
\cdot\Vert_{\infty}$ stands for the sup-norm and $\phi$ stands for the
solution of the torsional creep problem (\ref{pweight}) with
\[
\omega(x)=\max\left\{  \omega_{1}(x),\omega_{2}(x)\right\}  .
\]

Let $\lambda_{1}$ and $u_{1}$ be, respectively, the first eigenvalue and
positive eigenfunction associated to the weight $\omega_{1}$, with $\Vert
u_{1}\Vert_{\infty}=1$. Thus,%
\[
\left\{
\begin{array}
[c]{rcll}%
-\Delta_{p}u_{1} & = & \lambda_{1}\omega_{1}u_{1}^{p-1} & \text{in }\Omega\\
u_{1} & = & 0 & \text{on }\partial\Omega.
\end{array}
\right.
\]

We begin by considering problem (\ref{twotwo}) in the case $1<q<p$. In this
sublinear case, bounds for the sub-solution and ordering of the sub- and
super-solution pair follows from a simple lemma.

\begin{lemma}
\label{ltwo}It hold $\alpha\leq\lambda_{1}$ and
\[
\lambda_{1}\left(  \dfrac{\lambda}{\lambda_{1}}\right)  ^{\frac{p-1}{p-q}}%
\leq\alpha\left(  \dfrac{\lambda}{\alpha}\right)  ^{\frac{p-1}{p-q}}.
\]

\end{lemma}

\begin{proof}
Since $u_{1}=\phi=0$ on $\partial\Omega$ and%
\[
-\Delta_{p}u_{1}=\lambda_{1}\omega_{1}u_{1}^{p-1}\leq\lambda_{1}\omega
=-\Delta_{p}(\lambda_{1}^{\frac{1}{p-1}}\phi)\text{ \ in \ }\Omega
\]
it follows from the comparison principle applied to $u_{1}$ and $\lambda
_{1}^{\frac{1}{p-1}}\phi$ that%
\[
u_{1}\leq\lambda_{1}^{\frac{1}{p-1}}\phi\text{ \ in\ }\Omega.
\]
Therefore, the passing to maximum values yields
\[
1=\left\Vert u_{1}\right\Vert _{\infty}\leq\lambda_{1}^{\frac{1}{p-1}%
}\left\Vert \phi\right\Vert _{\infty}=\left(  \frac{\lambda_{1}}{\alpha
}\right)  ^{\frac{1}{p-1}}%
\]
implying that $\alpha\leq\lambda_{1}.$ Thus,
\[
\lambda_{1}\left(  \frac{\lambda}{\lambda_{1}}\right)  ^{\frac{p-1}{p-q}%
}=\lambda^{\frac{p-1}{p-q}}\left(  \frac{1}{\lambda_{1}}\right)  ^{\frac
{q-1}{p-q}}\leq\lambda^{\frac{p-1}{p-q}}\left(  \frac{1}{\alpha}\right)
^{\frac{q-1}{p-q}}=\alpha\left(  \frac{\lambda}{\alpha}\right)  ^{\frac
{p-1}{p-q}}.
\]

\vspace{-.8cm}
\end{proof}

\goodbreak

\begin{theorem}
\label{taux}Suppose $a+b=p-1$ and $1<q<p.$ If $\lambda>0$ and $0\leq
\beta<\dfrac{\alpha}{\mu^{b}}$, then $(\ref{twotwo})$ has at least one
positive solution $u\in C^{1,\tau}(\overline{\Omega})$ satisfying the bounds
\begin{equation}
\left(  \dfrac{\lambda}{\lambda_{1}}\right)  ^{\frac{1}{p-q}}u_{1}\leq
u\leq\left(  \dfrac{\lambda}{\alpha-\beta\mu^{b}}\right)  ^{\frac{1}{p-q}%
}\dfrac{\phi}{\left\Vert \phi\right\Vert _{\infty}}. \label{cotaq}%
\end{equation}

\end{theorem}

\begin{proof}
We consider $\overline{u}=M\dfrac{\phi}{\|\phi\|_{\infty}}$, where $M=\left(
\dfrac{\lambda}{\alpha-\beta\mu^{b}}\right)  ^{\frac{1}{p-q}}$. The definition
of $M$ yields
\begin{equation}
\label{eqM}\alpha M^{p-1}=\lambda M^{q-1}+\beta\mu^{b}M^{p-1}=\lambda
M^{q-1}+\beta\mu^{b}M^{a+b}.
\end{equation}

We also have%
\[
\left|  \nabla\overline{u}\right|  ^{b}=M^{b}\left(  \dfrac{\left|  \nabla
\phi\right|  }{\left\|  \phi\right\|  _{\infty}}\right)  ^{b}\leq\mu^{b}%
M^{b}.
\]

Thus,%
\begin{align*}
-\Delta_{p}\overline{u}  &  =\alpha M^{p-1}\omega(x)\\
&  \geq\lambda M^{q-1}\omega_{1}(x)+\beta\mu^{b}M^{a+b}\omega_{2}%
(x)\geq\lambda\omega_{1}(x)\overline{u}^{q-1}+\beta\omega_{2}(x)\overline
{u}^{a}\left\vert \nabla\overline{u}\right\vert ^{b}.
\end{align*}
Since $\overline{u}=0$ on $\partial\Omega$, we conclude that $\overline{u}$ is
a positive super-solution of (\ref{twotwo}).

We define $\underline{u}=\varepsilon u_{1}$, where $\varepsilon=\left(
\dfrac{\lambda}{\lambda_{1}}\right)  ^{\frac{1}{p-q}}$. We have
\begin{align*}
-\Delta_{p}\underline{u}  &  =\lambda_{1}\omega_{1}(x)\underline{u}^{p-1}\\
&  =\omega_{1}(x)(\varepsilon u_{1})^{q-1}\left(  \lambda_{1}\varepsilon
^{p-q}\right)  u_{1}^{p-q}\\
&  =\omega_{1}(x)\underline{u}^{q-1}\lambda u_{1}^{p-q}\leq\lambda\omega
_{1}(x)\underline{u}^{q-1}+\beta\omega_{2}(x)\underline{u}^{a}\left\vert
\nabla\underline{u}\right\vert ^{b},
\end{align*}
and since $\underline{u}=0$ in $\partial\Omega$, we conclude that
$\underline{u}$ is a positive sub-solution of (\ref{twotwo}).

By applying Lemma \ref{ltwo} and the comparison principle we obtain the
ordering $\underline{u}\leq\overline{u}$. In fact, we have%
\begin{align*}
-\Delta_{p}\underline{u}  &  =\lambda_{1}\omega_{1}(x)\left(  \frac{\lambda
}{\lambda_{1}}\right)  ^{\frac{p-1}{p-q}}u_{1}^{p-1}\\
&  \leq\lambda_{1}\left(  \frac{\lambda}{\lambda_{1}}\right)  ^{\frac
{p-1}{p-q}}\omega_{1}(x)\leq\alpha\left(  \frac{\lambda}{\alpha}\right)
^{\frac{p-1}{p-q}}\omega(x)\leq\alpha M^{p-1}\omega(x)=-\Delta_{p}\overline
{u},
\end{align*}
since (\ref{eqM}) implies that $\alpha M^{p-1}\geq\lambda M^{q-1}$ and,
therefore, $M\geq(\lambda/\alpha)^{1/(p-q)}$.
\end{proof}

\goodbreak

\begin{corollary}
Suppose $a+b<p-1$ and $1<q<p.$ If $\lambda>0$ and $\beta>0$, then
$(\ref{twotwo})$ has at least one positive solution $u\in C^{1,\tau}%
(\overline{\Omega})$ satisfying the bounds%
\[
\left(  \dfrac{\lambda}{\lambda_{1}}\right)  ^{\frac{1}{p-q}}u_{1}\leq u\leq
M^{\frac{1}{p-q}}\dfrac{\phi}{\left\Vert \phi\right\Vert _{\infty}},
\]
where $M>0$ satisfies the equation $\alpha M^{p-1}=\lambda M^{q-1}+\beta
\mu^{b}M^{a+b}.$
\end{corollary}

\begin{proof}
The proof above remains valid in this case. In fact, it easy to check that for
any $\lambda>0$ and $\beta>0$ the equation $\alpha M^{p-1}=\lambda
M^{q-1}+\beta\mu^{b}M^{a+b}$ has a unique positive solution $M.$ Hence,
$\overline{u}=M\dfrac{\phi}{\Vert\phi\Vert_{\infty}}$ is a super-solution for
(\ref{twotwo}). Moreover, $\underline{u}=\left(  \dfrac{\lambda}{\lambda_{1}%
}\right)  ^{\frac{1}{p-q}}u_{1}$ is a sub-solution for (\ref{twotwo}) and
$\underline{u}\leq\overline{u}$ since $\alpha M^{p-1}\geq\lambda M^{q-1}.\ $
\end{proof}

We now deal with the Dirichlet problem (\ref{twotwo}) in the case $q=p$ and
$a+b=p-1.$ Our approach considers $q\rightarrow p^{-}.$

\begin{theorem}
\label{mauxt}Suppose $a+b=p-1.$ For each $0\leq\beta<\dfrac{\alpha}{\mu^{b}}$,
there exist $\lambda_{\beta}>0$ and $u_{\beta}\in C^{1,\tau}\left(
\overline{\Omega}\right)  $ such that $0<u_{\beta}\leq1$ in $\Omega,$
$\alpha-\beta\mu^{b}\leq\lambda_{\beta}<\lambda_{1}$ and%
\begin{equation}
\left\{
\begin{array}
[c]{rcll}%
-\Delta_{p}u_{\beta} & = & \lambda_{\beta}\omega_{1}(x)u_{\beta}^{p-1}%
+\beta\omega_{2}(x)u_{\beta}^{a}|\nabla u_{\beta}|^{b} & \text{in }\Omega\\
u_{\beta} & = & 0 & \text{on }\partial\Omega.
\end{array}
\right.  \label{probp}%
\end{equation}

\end{theorem}

\begin{proof}
For each $0<q<p$ fixed, let us denote by $v_{q}$ the positive solution of
(\ref{twotwo}) given by Theorem \ref{taux}. Thus, multiplying the equation%
\[
-\Delta_{p}v_{q}=\lambda\omega_{1}(x)v_{q}^{q-1}+\beta\omega_{2}(x)v_{q}%
^{a}\left\vert \nabla v_{q}\right\vert ^{b}%
\]
by $\left(  \left\Vert v_{q}\right\Vert _{\infty}\right)  ^{1-p}$, we note
that $u_{q}:=\dfrac{v_{q}}{\left\Vert v_{q}\right\Vert _{\infty}}$ satisfies%
\[
\left\{
\begin{array}
[c]{rcll}%
-\Delta_{p}u_{q} & = & \lambda_{q}\omega_{1}(x)u_{q}^{q-1}+\beta\omega
_{2}(x)u_{q}^{a}|\nabla u_{q}|^{b} & \text{in }\Omega\\
u_{q} & = & 0 & \text{on }\partial\Omega,
\end{array}
\right.
\]
where $\lambda_{q}:=\dfrac{\lambda}{\left\Vert v_{q}\right\Vert _{\infty
}^{p-q}}$. It follows from (\ref{cotaq}) that%
\[
\dfrac{\lambda}{\lambda_{1}}\leq\left\Vert v_{q}\right\Vert _{\infty}%
^{p-q}\leq\dfrac{\lambda}{\alpha-\beta\mu^{b}}%
\]
and hence
\begin{equation}
0<\alpha-\beta\mu^{b}\leq\lambda_{q}\leq\lambda_{1}. \label{cotap}%
\end{equation}

Thus, since $0\leq u_{q}\leq1$ we have that
\begin{align*}
0  &  \leq\lambda_{q}\omega_{1}(x)u_{q}^{p-1}+\beta\omega_{2}(x)u_{q}%
^{a}\left\vert \nabla u_{q}\right\vert ^{b}\\
&  \leq\left\Vert \omega\right\Vert _{\infty}\left(  \lambda_{1}+\dfrac
{\alpha}{\mu^{b}}\left\vert \nabla u_{q}\right\vert ^{b}\right)  \leq
\Lambda\left(  1+\left\vert \nabla u_{q}\right\vert ^{p}\right)  ,
\end{align*}
for some positive constant $\Lambda$ which does not depend on $q$. So, we can
apply Theorem \ref{Liebb} to guarantee that $u_{q}\in C^{1,\tau}\left(
\overline{\Omega}\right)  $ and that $\left\Vert u_{q}\right\Vert _{1,\tau
}\leq C$ for some positive constant $C$ which does not depend on $q$.

Therefore, by taking a sequence $q_{n}\rightarrow p^{-}$, the compactness of
the immersion $C^{1,\tau}\left(  \overline{\Omega}\right)  \hookrightarrow
C^{1}\left(  \overline{\Omega}\right)  $ implies that, passing to a
subsequence, $u_{q_{n}}\rightarrow u_{\beta}$ in $C^{1}\left(  \overline
{\Omega}\right)  $ and $\lambda_{q_{n}}\rightarrow\lambda_{\beta}$. Thus, the
continuity of the operator $-\Delta_{p}^{-1}$ yields that $\lambda_{\beta}$
and $u_{\beta}$ satisfy (\ref{probp}). Moreover, it follows from (\ref{cotap})
that $0<\alpha-\beta\mu^{b}\leq\lambda_{\beta}\leq\lambda_{1}$.

Since $u_{\beta}>0$ we must have $\lambda_{\beta}<\lambda_{1}.$ This follows
from the following fact: if $\lambda\geq\lambda_{1}$, then $u\equiv0$ is the
only nonnegative solution of (\ref{probp}). Indeed, if $u$ is a nonnegative
solution of (\ref{probp}) for some $\lambda\geq\lambda_{1}$, then we can
write
\[
-\Delta_{p}u=\lambda\omega_{1}(x)u^{p-1}+\beta\omega_{2}(x)u^{a}\left\vert
\nabla u\right\vert ^{b}=\lambda_{1}\omega_{1}(x)u^{p-1}+h(x)
\]
where
\[
h(x)=\beta\omega_{2}(x)u^{a}\left\vert \nabla u\right\vert ^{b}+\left(
\lambda-\lambda_{1}\right)  \omega_{1}(x)u^{p-1}\geq0.
\]

Hence, by a consequence of Picone's identity (see \cite{Allegretto} and also
\cite{bez}, Lemma 8.1) we obtain $h\equiv0$ and thus, $u\equiv0$ in $\Omega.$
\end{proof}

\section{The abstract problem}

\label{main} From now on we consider the following problem
\begin{equation}
\left\{
\begin{array}
[c]{rcll}%
-\Delta_{p}u & = & f(x,u,\nabla u) & \text{in }\Omega\\
u & = & 0 & \text{on }\partial\Omega.
\end{array}
\right.  \label{abst2}%
\end{equation}

We begin stating precisely our hypotheses on the nonlinearity $f$. For this,
for a chosen (continuous) weight function $\omega\neq0$, let $\phi\in
C^{1,\tau}(\bar{\Omega})\cap W_{0}^{1,p}(\Omega)$ be the solution of the
problem
\[
\left\{
\begin{array}
[c]{rcll}%
-\Delta_{p}\phi & = & \omega & \text{in }\Omega\\
\phi & = & 0 & \text{on }\partial\Omega.
\end{array}
\right.
\]
As before, we have that $\phi\in C^{1,\tau}(\overline{\Omega})$ and $\phi>0$
in $\Omega$.

Let $\lambda_{1}$ and $u_{1}$ be the first eigenvalue and eigenfunction of the
$p$-Laplacian with weight $\omega,$ that is,
\begin{equation}
\left\{
\begin{array}
[c]{rcll}%
-\Delta_{p}u_{1} & = & \lambda_{1}\omega u_{1}^{p-1} & \text{in }\Omega\\
u_{1} & = & 0 & \text{on }\partial\Omega.
\end{array}
\right.
\end{equation}
We assume that $u_{1}$ is positive and $\Vert u_{1}\Vert_{\infty}=1$.

We set, as in our approach in the previous section,
\begin{equation}
\alpha=\Vert\phi\Vert_{\infty}^{1-p} \label{k1}%
\end{equation}
and
\begin{equation}
\mu=\frac{\Vert\nabla\phi\Vert_{\infty}}{\Vert\phi\Vert_{\infty}}. \label{mu}%
\end{equation}

\begin{remark}
\label{rm3}\textrm{It follows from Lemma \ref{ltwo} that $\alpha\leq
\lambda_{1}.$ However, is not difficult to verify that $\alpha<\lambda_{1}$
(\cite[Section 8]{bez}). }
\end{remark}

We assume that, \emph{besides} (H1), the \emph{continuous} nonlinearity $f$
satisfies, for an arbitrary constant $M>0$,

\begin{enumerate}
\item[(H2)] $\displaystyle{\lim_{u\rightarrow0^{+}}\frac{f(x,u,v)}{u^{p-1}}%
}\geq\lambda_{1}\omega(x),\quad(x,v)\in\overline{\Omega}\times B_{\mu M}%
\quad\text{(uniformly)}$, \vspace{.1cm}

where $B_{\mu M}=\left\{  v\in\mathbb{R}^{N}\,:\,|v|\leq\mu M\right\}  $;

\item[(H3)] $0\leq f(x,u,v)\leq\alpha\omega(x)M^{p-1},\qquad(x,u,v)\in
\overline{\Omega}\times[0,M]\times B_{\mu M}$.
\end{enumerate}

In Section \ref{App}, hypotheses (H2) and (H3) are interpreted in a particular situation.

As mentioned before, hypothesis (H1) might be changed for any hypothesis that
produces a solution of (\ref{abst2}) from an ordered sub- and super-solution
pair of this problem. However, we do believe that adequate arguments of
extension and truncation might produce (H1) from (H3). In fact, if we have a
priori estimates for the gradient of uniformly bounded solutions of
(\ref{abst2}), the hypothesis (H1) is not necessary. For example, in
\cite{BEFZ}, where the case $f(x,u,v)=\omega(x)f(u,|v|)$ and $\Omega=B_{r}$ (a
ball) is handled, all solutions $u$ of (\ref{abst2}) such that $\Vert
u\Vert_{\infty}\leq M$ also satisfy $\Vert\nabla u\Vert_{\infty}\leq\mu M$.

We now state the main result of the paper concerning problem(\ref{abst2}).

\begin{theorem}
\label{maint} Define $u_{\epsilon}=\epsilon u_{1}$. If the nonlinearity $f$
satisfies $(\mathrm{H}1)-(\mathrm{H}3)$, the Dirichlet problem (\ref{abst2})
has at least one positive solution $u\in C^{1,\alpha}(\bar{\Omega})\cap
W_{0}^{1,p}(\Omega)$ satisfying the bounds
\begin{equation}
0<u_{\epsilon}\leq u\leq\frac{M\phi}{\Vert\phi\Vert_{\infty}}\quad\text{in
}\Omega, \label{loc}%
\end{equation}
for all $\epsilon>0$ sufficiently small.
\end{theorem}

We prove this theorem in a sequence of simple results.

\begin{proposition}
The function $U:=\frac{M\phi}{\Vert\phi\Vert_{\infty}}\in C^{1,\tau}%
(\bar{\Omega})\cap W_{0}^{1,p}(\Omega)$ is a super-solution for the problem
(\ref{abst2}).
\end{proposition}

\begin{proof}
Of course, we have $0\leq U\leq M$ and $0\leq|\nabla U|=\frac{M|\nabla\phi
|}{\|\phi\|_{\infty}}\leq\mu M$. Thus, it follows from (H3) that $\alpha
M^{p-1}\omega\geq f(x,U,\nabla U)$ and
\[
-\Delta_{p}U=-\Delta_{p}\left(  \frac{M\phi}{\|\phi\|_{\infty}}\right)
=\alpha M^{p-1}\omega\geq f(x,U,\nabla U).
\]

Since $U=0$ on $\partial\Omega$, we are done.
\end{proof}

\begin{proposition}
Define $u_{\epsilon}=\epsilon u_{1}$ for $\epsilon>0$. Then, for $\epsilon$
sufficiently small, $u_{\epsilon}$ is a sub-solution for problem (\ref{abst2}).
\end{proposition}

\begin{proof}
For all $0<\epsilon\leq{\frac{\mu M}{\|\nabla u_{1}\|_{\infty}}}$ we have
\begin{equation}
\label{equa1}0\leq u_{\epsilon}=\epsilon u_{1}\leq\epsilon\|u_{1}\|_{\infty
}=\epsilon
\end{equation}
and
\begin{equation}
\label{equa2}0\leq|\nabla u_{\epsilon}|=\epsilon|\nabla u_{1}|\leq
\epsilon\|\nabla u_{1}\|_{\infty}\leq\mu M.
\end{equation}

Now, it follows from (H2) the existence of $\epsilon_{0}>0$ such that
\[
f(x,u,v)\geq\lambda_{1}\omega(x)u^{p-1}\quad\text{for all }\ 0\leq
u\leq\epsilon_{0}\quad\text{and }\ (x,v)\in\overline{\Omega}\times B_{\mu M}%
\]
In particular, if $0<\epsilon\leq\min\left\{  \epsilon_{0},\displaystyle{\frac
{\mu M}{\Vert\nabla u_{1}\Vert_{\infty}}}\right\}  $, then
\[
f(x,u_{\epsilon},\nabla u_{\epsilon})\geq\lambda_{1}\omega(x)u_{\epsilon
}^{p-1}\quad\text{for all }
x\in\overline{\Omega},
\]
that is,%
\begin{equation}
-\Delta_{p}u_{\epsilon}=\lambda_{1}\omega(x)u_{\epsilon}^{p-1}\leq f\left(
x,u_{\epsilon},\nabla u_{\epsilon}\right)  \quad\text{in }\Omega. \label{aux}%
\end{equation}

Since $u_{\epsilon}=0$ on $\partial\Omega$, $u_{\epsilon}$ is a sub-solution
for (\ref{abst2}).
\end{proof}

\noindent\textbf{Proof of the Theorem.} It follows from Theorem \ref{TBoc}
that we only need to verify that $u_{\epsilon}\leq U$ for $\epsilon>0$
sufficiently small.

Taking $\epsilon<\min\left\{  \epsilon_{0},M,\displaystyle\frac{\mu M}%
{\Vert\nabla u_{1}\Vert_{\infty}}\right\}  $, we have
\[
-\Delta_{p}u_{\epsilon}\leq f(x,u_{\epsilon},\nabla u_{\epsilon})\leq\alpha
M^{p-1}\omega(x)=-\Delta_{p}\left(  \frac{M\phi}{\Vert\phi\Vert_{\infty}%
}\right)  =-\Delta_{p}U\quad\text{in }\Omega.
\]
The first inequality follows from (\ref{aux}), while the second inequality
follows from (H3) by applying (\ref{equa1}) and (\ref{equa2}). Since
$u_{\epsilon}=0=U$ on $\partial\Omega$, $u_{\epsilon}\leq U$ is a consequence
of the comparison principle. $\hfill\Box$\vspace{0.5cm}

\section{Applications}

\label{App} An abstract example of a nonlinearity $f$ satisfying our
hypotheses is given by
\[
f(x,u,v)=\omega(x)g\left(  u,|v|\right)  ,
\]
where $\omega$ is a continuous weight function defined in $\Omega$ and
$g(u,t)$ is a continuous function satisfying
\begin{equation}
g(u,t)\leq C(|u|)(1+t^{p})\quad\text{for all }\ (u,t)\in\mathbb{R}%
\times(0,\infty). \tag{H3}%
\end{equation}
and also
\begin{equation}
g(u,t)\geq\lambda_{1}u^{p-1},\quad\text{for all }\ (u,t)\in\lbrack
0,\epsilon]\times\lbrack0,\mu M] \label{g1}%
\end{equation}
for some $\epsilon>0$,
\begin{equation}
0\leq g(u,t)\leq\alpha M^{p-1},\quad\text{for all }\ (u,t)\in\lbrack
0,M]\times\lbrack0,\mu M] \label{g2}%
\end{equation}
for some $M>0$. Note that (\ref{g1}) and (\ref{g2}) are hypotheses (H1) and
(H2) for this particular, abstract example.

Geometrically, the 2-variable function $g(u,t)$ has its graph passing through
a ``rectangular box with a small step'' in its floor, formed by the surfaces
\[
\left\{
\begin{array}
[c]{l}%
z=\lambda_{1}u^{p-1},\\
0\leq u\leq\epsilon,\\
0\leq t\leq\mu M
\end{array}
\right.  \quad\text{and}\quad\left\{
\begin{array}
[c]{l}%
u=\epsilon,\\
0\leq z\leq\lambda_{1}\epsilon^{p-1},\\
0\leq t\leq\mu M.
\end{array}
\right.
\]

Figure 1 illustrates such a box for the case $p=2$ and Figure 2 shows the
region obtained by sectioning it by the plane $t\equiv c$ for $c\in[0,\mu M].$

\begin{figure}[h]
\begin{center}
{\includegraphics[width=5cm]{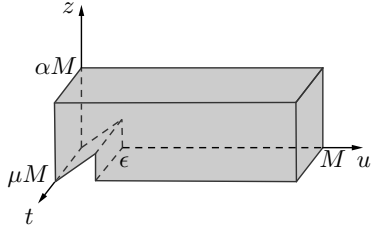}}
\end{center}
\par
\vspace{-.5cm}\caption{For $c\in[0,\mu M]$, the graph of $g(u,c)$ passes
through a ``box with a small step'' in its floor. (Here $p=2$.)}%
\end{figure}

\begin{figure}[th]
\begin{center}
{\includegraphics[width=5cm]{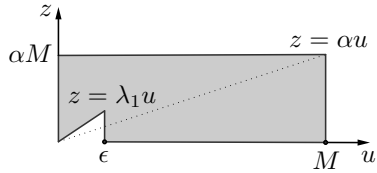}}
\end{center}
\par
\vspace{-.5cm}\caption{The (orthogonal) projection of the graph of $g(u,c)$ on
the $uz$ plane passes through the gray area. (Here $p=2$).}%
\end{figure}

It is noteworthy that the box can be made sufficiently large by increasing $M$
and its step can be made sufficiently small by decreasing $\epsilon$.
Moreover, $g(u,t)$ could be zero at several values in this box.

From the proofs presented above we can see that, once fixed the weight
$\omega$, if the graph of $g(u,t)$ passes through such a box, i.e., if $g$
satisfies (\ref{g1}) and (\ref{g2}) and a growth condition like (H1), then the
Dirichlet problem
\[
\left\{
\begin{array}
[c]{rcll}%
-\Delta_{p}u & = & \omega(x)g(u,|\nabla u|), & \text{in }\Omega\\
u & = & 0, & \text{on }\partial\Omega,
\end{array}
\right.
\]
has a positive solution $u$ bounded by two functions determined only by
$\omega$ and $\Omega$.

\begin{example}
\label{ex1} \textrm{We consider the problem
\begin{equation}
\label{p2}\left\{
\begin{array}
[c]{rcll}%
-\Delta_{p}u & = & \lambda\omega(x)u^{q-1}\left(  1+|\nabla u|^{p}\right)  &
\text{in }\Omega\\
u & = & 0 & \text{on }\partial\Omega,
\end{array}
\right.
\end{equation}
where $\omega$ is a positive weight function, $1<q<p$ and $\lambda\in
(0,\infty)$. We will show that there exists $\lambda_{*}$ (to be specified in
the sequence) so that problem $(\ref{p2})$ has a positive solution for all
$\lambda\in(0,\lambda_{*}]$. }

\textrm{The nonlinearity $f(x,u,v)=\lambda\omega(x)u^{q-1}\left(
1+|v|^{p}\right)  $ satisfies $(\mathrm{H}1)$ for all $\lambda$. Moreover, it
satisfies $(\mathrm{H}2)$ uniformly for all $v\in\mathbb{R}^{N}$, since
\[
\lim_{u\to0^{+}}\frac{\lambda\omega(x)u^{q-1}\left(  1+|v|^{p}\right)
}{u^{p-1}}\geq\lim_{u\to0^{+}}\frac{\lambda\omega(x)u^{q-1}}{u^{p-1}}\geq
\frac{\lambda\inf_{\Omega}\omega}{u^{p-q}}=\infty,
\]
}

\textrm{In order to satisfy $(\mathrm{H}3)$, we must have
\[
\lambda M^{q-1}(1+(\mu M)^{p})\leq\alpha M^{p-1}.
\]
}

\textrm{So, defining the function $H\colon[0,\infty)\to[0,\infty]$ by
$H(M)=M^{q-p}(1+\mu^{p} M^{p})$, the last inequality is equivalent to
\[
H(M)\leq\frac{\alpha}{\lambda}.
\]
We have
\[
\lim_{M\to0^{+}}H(M)=\infty=\lim_{M\to\infty}H(M),
\]
and the function $H$ has a unique critical point $M_{*}$, given by
\[
\mu^{p} M^{p}_{*}=\frac{p}{q}-1,
\]
where $H$ assumes its minimum value
\[
H(M_{*})=M^{q-p}(1+\mu^{p}M_{*}^{p})=\frac{1}{\mu^{q-p}}\left(  \frac{p}%
{q}-1\right)  ^{\frac{q-p}{p}}\left(  \frac{p}{q}\right)  .
\]
}

\textrm{By setting
\[
\lambda_{*}=\frac{\alpha}{\mu^{p-q}}\left(  \frac{p}{q}-1\right)  ^{\frac
{p-q}{p}}\left(  \frac{q}{p}\right)  ,
\]
we obtain
\[
H(M_{*})=\frac{\alpha}{\lambda_{*}}.
\]
}

\textrm{So, condition $(\mathrm{H}3)$ is satisfied by the nonlinearity
$\lambda\omega(x)u^{q-1}\left(  1+|v|^{p}\right)  $ if $\frac{\alpha}%
{\lambda_{*}}\leq\frac{\alpha}{\lambda}$, that is,
\[
0<\lambda\leq\lambda^{*}.
\]
}

\textrm{For a fixed $\lambda\in(0,\lambda_{*}]$, in order to obtain estimates
for the solution $u_{\lambda}$ of $(\ref{p2})$, we define $\varepsilon
:=\left(  \frac{\lambda}{\lambda_{1}}\right)  ^{\frac{1}{p-q}}$, and note that
$\underline{u}_{\lambda}:=\varepsilon u_{1}$ is a sub-solution of this
problem:
\begin{align*}
-\Delta_{p}\underline{u}_{\lambda}  &  =\lambda_{1}\underline{u}_{\lambda
}^{p-1}\omega\\
&  =\lambda_{1}\underline{u}_{\lambda}^{q-1}\omega\underline{u}_{\lambda
}^{p-q}\leq\varepsilon^{p-q}\lambda_{1}\underline{u}_{\lambda}^{q-1}%
\omega=\lambda\underline{u}_{\lambda}^{q-1}\omega\leq\lambda\underline
{u}_{\lambda}^{q-1}\omega\left(  1+\left|  \nabla\underline{u}_{\lambda
}\right|  ^{p}\right)
\end{align*}
($u_{1}$ denotes, as before, the positive solution of $-\Delta_{p}%
u_{1}=\lambda_{1}u^{p-1}\omega$, with $\|u_{1}\|_{\infty}=1$.) }

\textrm{The sub-solution $\underline{u}_{\lambda}$ and the super-solution
$U=M_{*}\phi/\|\phi\|_{\infty}$ (given by Theorem $\ref{maint}$) are ordered,
if we choose $\varepsilon$ such that $\lambda_{1}\varepsilon^{p-1}\leq\alpha
M_{*}^{p-1}$. In fact, follows from the comparison principle that
\[
-\Delta_{p}\underline{u}_{\lambda}=\lambda_{1}(\varepsilon u_{1})^{p-1}%
\omega\leq\lambda_{1}\varepsilon^{p-1}\omega\leq\alpha M_{*}^{p-1}%
\omega=-\Delta_{p}U.
\]
From the bounds $\|u_{\epsilon}\|_{\infty}<\|u_{\lambda}\|_{\infty}%
\leq\|U\|_{\infty}$, we conclude that
\[
\left(  \frac{\lambda}{\lambda_{1}}\right)  ^{\frac{1}{p-q}}\leq\left\|
u_{\lambda}\right\|  _{\infty}\leq\frac{1}{\mu}\left(  \frac{p}{q}-1\right)
^{\frac{1}{p}}.
\]
}
\end{example}

\begin{example}
\label{ex2} \textrm{The problem
\begin{equation}
\left\{
\begin{array}
[c]{rcll}%
-\Delta_{p} u & = & \lambda f(x,u)+|\nabla u|^{p}, & \text{in }\ \Omega\\
u & = & 0, & \text{on }\ \partial\Omega,
\end{array}
\right. \label{dir}%
\end{equation}
where $\lambda>0$ is a parameter and $f(x,u)$ is a Carath\'{e}odory function
satisfying%
\begin{equation}
c_{0}u^{q-1}\leq f(x,u)\leq c_{1}u^{q-1},\quad\text{for all}\ (x,t)\in
\overline{\Omega}\times[0,\infty)\label{between}%
\end{equation}
for positive constants $c_{0}$ and $c_{1}$ was treated in $\cite{ILS}$ for the
cases $q>p$, $q=p$ and $1<q<p$. }

\textrm{In the case $1<q<p$, they proved the existence of a positive value
$\overline{\Lambda}$ such that the problem has at least two positive solutions
if $0<\lambda<\overline{\Lambda}$, at least one positive solution if
$\lambda=\overline{\Lambda}$ and no positive solution if $\lambda
>\overline{\Lambda}$, a result the resembles one of the first steps in the
study of the classic Ambrosetti-Prodi problem. }

\textrm{In that paper, by making the change of variable $w=e^{\frac{u}{p-1}%
}-1$, problem $(\ref{dir})$ was transformed into another problem, whose
nonlinearity $h(x,w)$ does not depend on the gradient of $w$. Then, the
existence of solution was obtained by applying the sub- and super-solution
method to the transformed problem. (See also the variational approach for this
problem in \cite{ILU}.) }

\textrm{A direct application of Theorem $\ref{maint}$ gives the existence of
(at least) one positive solution for
\[
0<\lambda\leq\lambda_{*}:=\frac{1}{c_{1}}\left(  \frac{p-q}{\mu^{p}}\right)
^{p-q}\left(  \frac{\alpha}{p-q+1}\right)  ^{p-q}.
\]
The details are very similar to that of Example $\ref{ex1}$. }
\end{example}

\end{document}